\documentclass[11pt,leqno]{article}
\usepackage[margin=1in]{geometry} 
\usepackage{amssymb,amsfonts,amsmath,bbm,mathrsfs,stmaryrd,mathtools}
\usepackage{xcolor}
\usepackage{url}

\usepackage{extarrows}

\usepackage[shortlabels]{enumitem}
\usepackage{tensor}

\usepackage{xr}
\usepackage[T1]{fontenc}

\usepackage[colorlinks,
linkcolor=black!75!red,
citecolor=blue,
pdftitle={},
pdfproducer={pdfLaTeX},
pdfpagemode=None,
bookmarksopen=true,
bookmarksnumbered=true,
backref=page]{hyperref}

\usepackage{tikz}
\usetikzlibrary{arrows,calc,decorations.pathreplacing,decorations.markings,intersections,shapes.geometric,through,fit,shapes.symbols,positioning,decorations.pathmorphing}

\usepackage{braket}

\usepackage[amsmath,thmmarks,hyperref]{ntheorem}
\usepackage{cleveref}

\newcommand\nthalias[1]{\AddToHook{env/#1/begin}{\crefalias{lemma}{#1}}}

\nthalias{definition}
\nthalias{example}
\nthalias{examples}
\nthalias{remark}
\nthalias{remarks}
\nthalias{convention}
\nthalias{notation}
\nthalias{construction}
\nthalias{theoremN}
\nthalias{propositionN}
\nthalias{corollaryN}
\nthalias{lemma}
\nthalias{proposition}
\nthalias{corollary}
\nthalias{theorem}
\nthalias{conjecture}
\nthalias{question}
\nthalias{assumption}

\creflabelformat{enumi}{#2#1#3}

\crefname{section}{Section}{Sections}
\crefformat{section}{#2Section~#1#3} 
\Crefformat{section}{#2Section~#1#3} 

\crefname{subsection}{\S}{\S\S}
\AtBeginDocument{%
  \crefformat{subsection}{#2\S#1#3}%
  \Crefformat{subsection}{#2\S#1#3}%
}

\crefname{subsubsection}{\S}{\S\S}
\AtBeginDocument{%
  \crefformat{subsubsection}{#2\S#1#3}%
  \Crefformat{subsubsection}{#2\S#1#3}%
}

\theoremstyle{plain}

\newtheorem{lemma}{Lemma}[section]
\newtheorem{proposition}[lemma]{Proposition}
\newtheorem{corollary}[lemma]{Corollary}
\newtheorem{theorem}[lemma]{Theorem}

\theoremstyle{plain}
\theoremnumbering{Alph}
\newtheorem{theoremN}{Theorem}

\theoremstyle{plain}
\theorembodyfont{\upshape}
\theoremsymbol{\ensuremath{\blacklozenge}}

\newtheorem{definition}[lemma]{Definition}

\newtheorem{remark}[lemma]{Remark}
\newtheorem{remarks}[lemma]{Remarks}

\crefname{definition}{definition}{definitions}
\crefformat{definition}{#2definition~#1#3} 
\Crefformat{definition}{#2Definition~#1#3} 

\crefname{ex}{example}{examples}
\crefformat{example}{#2example~#1#3} 
\Crefformat{example}{#2Example~#1#3} 

\crefname{exs}{example}{examples}
\crefformat{examples}{#2example~#1#3} 
\Crefformat{examples}{#2Example~#1#3} 

\crefname{remark}{remark}{remarks}
\crefformat{remark}{#2remark~#1#3} 
\Crefformat{remark}{#2Remark~#1#3} 

\crefname{remarks}{remark}{remarks}
\crefformat{remarks}{#2remark~#1#3} 
\Crefformat{remarks}{#2Remark~#1#3} 

\crefname{convention}{convention}{conventions}
\crefformat{convention}{#2convention~#1#3} 
\Crefformat{convention}{#2Convention~#1#3} 

\crefname{notation}{notation}{notations}
\crefformat{notation}{#2notation~#1#3} 
\Crefformat{notation}{#2Notation~#1#3} 

\crefname{table}{table}{tables}
\crefformat{table}{#2table~#1#3} 
\Crefformat{table}{#2Table~#1#3}

\crefname{lemma}{lemma}{lemmas}
\crefformat{lemma}{#2lemma~#1#3} 
\Crefformat{lemma}{#2Lemma~#1#3} 

\crefname{proposition}{proposition}{propositions}
\crefformat{proposition}{#2proposition~#1#3} 
\Crefformat{proposition}{#2Proposition~#1#3} 

\crefname{propositionN}{proposition}{propositions}
\crefformat{propositionN}{#2proposition~#1#3} 
\Crefformat{propositionN}{#2Proposition~#1#3} 

\crefname{corollary}{corollary}{corollaries}
\crefformat{corollary}{#2corollary~#1#3} 
\Crefformat{corollary}{#2Corollary~#1#3} 

\crefname{corollaryN}{corollary}{corollaries}
\crefformat{corollaryN}{#2corollary~#1#3} 
\Crefformat{corollaryN}{#2Corollary~#1#3} 

\crefname{theorem}{theorem}{theorems}
\crefformat{theorem}{#2theorem~#1#3} 
\Crefformat{theorem}{#2Theorem~#1#3} 

\crefname{theoremN}{theorem}{theorems}
\crefformat{theoremN}{#2theorem~#1#3} 
\Crefformat{theoremN}{#2Theorem~#1#3} 

\crefname{enumi}{}{}
\crefformat{enumi}{#2#1#3}
\Crefformat{enumi}{#2#1#3}

\crefname{assumption}{assumption}{Assumptions}
\crefformat{assumption}{#2assumption~#1#3} 
\Crefformat{assumption}{#2Assumption~#1#3} 

\crefname{construction}{construction}{Constructions}
\crefformat{construction}{#2construction~#1#3} 
\Crefformat{construction}{#2Construction~#1#3} 

\crefname{question}{question}{Questions}
\crefformat{question}{#2question~#1#3} 
\Crefformat{question}{#2Question~#1#3} 

\crefname{equation}{}{}
\crefformat{equation}{(#2#1#3)} 
\Crefformat{equation}{(#2#1#3)}

\numberwithin{equation}{section}

\theoremstyle{nonumberplain}
\theoremsymbol{\ensuremath{\blacksquare}}

\newtheorem{proof}{Proof}
\newcommand\pf[1]{\newtheorem{#1}{Proof of \Cref{#1}}}

\newcommand\bA{{\mathbb A}}

\newcommand\bC{{\mathbb C}}
\newcommand\bD{{\mathbb D}}

\newcommand\bG{{\mathbb G}}
\newcommand\bH{{\mathbb H}}

\newcommand\bK{{\mathbb K}}
\newcommand\bL{{\mathbb L}}

\newcommand\bP{{\mathbb P}}
\newcommand\bQ{{\mathbb Q}}
\newcommand\bR{{\mathbb R}}
\newcommand\bS{{\mathbb S}}
\newcommand\bT{{\mathbb T}}

\newcommand\bZ{{\mathbb Z}}

\newcommand\cC{{\mathcal C}}

\newcommand\fb{{\mathfrak b}}

\newcommand\fg{{\mathfrak g}}

\newcommand\fl{{\mathfrak l}}

\newcommand\fs{{\mathfrak s}}
\newcommand\ft{{\mathfrak t}}

\newcommand\fz{{\mathfrak z}}

\DeclareMathOperator{\Ad}{Ad}

\DeclareMathOperator{\aspn}{\mathrm{affspan}}
\DeclareMathOperator{\lk}{\mathrm{lk}}

\DeclareMathOperator{\Rep}{\cat{Rep}}
\DeclareMathOperator{\Irr}{\cat{Irr}}

\DeclareMathOperator{\rk}{\mathrm{rk}}
\DeclareMathOperator{\GL}{GL}

\DeclareMathOperator{\SU}{SU}

\DeclareMathOperator{\U}{U}

\DeclareMathOperator{\spc}{sp}

\newcommand{\cat}[1]{\textsc{#1}}

\newcommand{\xrightarrowdbl}[2][]{%
  \xrightarrow[#1]{#2}\mathrel{\mkern-14mu}\rightarrow
}

\title{Eigenvalue selectors for representations of compact connected groups}
\author{Alexandru Chirvasitu}

\begin{document}

\date{}

\newcommand{\Addresses}{{
  \bigskip
  \footnotesize

  \textsc{Department of Mathematics, University at Buffalo}
  \par\nopagebreak
  \textsc{Buffalo, NY 14260-2900, USA}  
  \par\nopagebreak
  \textit{E-mail address}: \texttt{achirvas@buffalo.edu}


}}

\maketitle

\begin{abstract}
  A representation $\rho$ of a compact group $\mathbb{G}$ selects eigenvalues if there is a continuous circle-valued map on $\mathbb{G}$ assigning an eigenvalue of $\rho(g)$ to every $g\in \mathbb{G}$. For every compact connected $\mathbb{G}$, we characterize the irreducible $\mathbb{G}$-representations which select eigenvalues as precisely those annihilating the intersection $Z_0(\mathbb{G})\cap \mathbb{G}'$ of the connected center of $\mathbb{G}$ with its derived subgroup. The result applies more generally to finite-spectrum representations isotypic on $Z_0(\mathbb{G})$, and recovers as applications (noted in prior work) the existence of a continuous eigenvalue selector for the natural representation of $\mathrm{SU}(n)$ and the non-existence of such a selector for $\mathrm{U}(n)$. 
\end{abstract}


\vspace{.5cm}

\noindent{MSC 2020: 22E46; 22C05; 55R10; 47A10; 17B10; 17B22; 54D05; 20F55
  
  
}


\section*{Introduction}

The original motivation for the present paper lies in remarks adjacent to work carried out in \cite{2501.06840v1}, on and around classifying \emph{spectrum-shrinking} maps
\begin{equation*}
  \text{$n\times n$ unitary group}
  =:
  \U(n)
  \xrightarrow[\quad\text{continuous}\quad]{\quad\varphi\quad}
  M_n
  :=
  \text{$n\times n$ complex matrices}.
\end{equation*}
\cite[Corollary 1.2(d)]{2501.06840v1} places severe constraints on such a map, which as an immediate consequence (in any case easily verified directly) imply that $\U(n)$ does not admit a \emph{continuous eigenvalue selector}
\begin{equation*}
  \U(n)
  \xrightarrow[\quad\text{continuous}\quad]{\quad\varphi\quad}
  \bS^1
  ,\quad
  \varphi(U)\in\sigma(U):=\text{spectrum of $U$}
  ,\quad
  \forall U\in \U(n)
\end{equation*}
(see also \cite[Example 2]{LiZhang} for the analogous claim for $M_n$ rather than $\U(n)$). By contrast, the special unitary groups $\SU(n)$ \emph{do} admit such continuous eigenvalue selectors. As sketched in \cite[Remark 1.4(3)]{2501.06840v1}, this is ultimately a consequence of Morton's description \cite[Theorem]{MR210096} of the \emph{symmetric power}
\begin{equation}\label{eq:quot.is.bdl}
  \left(\bS^1\right)^{[n]}
  :=
  \left(\bS^1\right)^n/\left(\text{symmetric group }S_n\right)
  \overset{\text{\cite[Proposition IV.2.6]{btd_lie_1995}}}{\cong}
  \U(n)/\text{adjoint action}
\end{equation}
as a fibration over $\bS^1$ with simplex fibers. Here, we place these remarks in a broader context of characterizing representations of compact connected (perhaps non-Lie) groups admitting such selectors. 

In the sequel $\Rep(\bG)=\Rep_{\bC}(\bG)$ denotes the category of complex representations of a compact group $\bG$ that are unions of finite-dimensional ones. Equivalently, by semisimplicity \cite[Proposition II.1.9]{btd_lie_1995}, these are exactly the (possibly infinite) direct sums of irreducible (finite-dimensional) $\bG$-representations. By standard Peter-Weyl theory (e.g. \cite[Theorems 3.51 and 4.22]{hm5}) $\Rep(\bG)$ is also equivalent to
\begin{itemize}
\item the category of \emph{unitary} $\bG$-representations on (possibly infinite-dimensional) complex Hilbert spaces;

\item or the category of representations on arbitrary complex Banach spaces $E$,
\end{itemize}
in both cases the continuity constraint on a representation $\bG\xrightarrow{\rho}\GL(E)$ being
\begin{equation*}
  \forall v\in E
  \quad:\quad
  \bG\ni g\xmapsto{\quad}\rho(g)v\in E
  \text{ is continuous}.
\end{equation*}

This in place, the main operative notions alluded to above are as follows. 

\begin{definition}\label{def:eig.sel}
  \begin{enumerate}[(1),wide]
  \item Let $\bG$ be a compact group. A \emph{continuous eigenvalue selector} for a representation $\rho\in \mathrm{Rep}(\bG)$ or a \emph{continuous $\rho$-selector} is a continuous function
    \begin{equation*}
      \bG\xrightarrow{\quad\varphi\quad}\bS^1
      \quad\text{with}\quad
      \varphi(g)\in \spc\rho(g)
      ,\quad
      \forall g\in \bG.
    \end{equation*}
    If $\rho$ admits one, we also say that it \emph{selects eigenvalues}.

  \item $\bG$ \emph{selects eigenvalues} (or \emph{has continuous eigenvalue selectors}) if every $\rho\in \Rep(\bG)$ does.
  \end{enumerate}  
\end{definition}

In the language just introduced, the two motivating remarks mentioned above are:
\begin{enumerate}[(a),wide]
\item the natural $n$-dimensional representation of $\SU(n)$ selects eigenvalues;

\item while the same representation of $\U(n)$ does not. 
\end{enumerate}

\Cref{thn:sel.iff.z0g} below specializes to a characterization of the irreducible $\bG$-representations admitting continuous eigenvalue selectors for any compact connected $\bG$, recovering the two observations as straightforward consequences. 

We follow fairly standard conventions (e.g. \cite[p.xi]{hm5}) in denoting by $\bG_0$ the identity connected component of a topological group $\bG$. Similarly,
\begin{itemize}[wide]
\item $Z_0(\bG):=Z(\bG)_0$ \cite[p.xiv, Theorem 23]{hm5} is the identity component of the center $Z(\bG)$;
  
\item and $\bG'$ is the \emph{commutator} (or \emph{derived}) \emph{subgroup} \cite[p.xiii, Definition 15]{hm5}, i.e. the subgroup generated by commutators
  \begin{equation*}
    [x,y]:=xyx^{-1}y^{-1}
    ,\quad
    x\,y\in \bG.
  \end{equation*}
  Recall \cite[Theorem 9.2]{hm5} that for compact connected groups $\bG$ the derived subgroup is in fact closed, and consists of commutators. 
\end{itemize}

Denote by $\Irr(\bG)$ the set of isomorphism classes of simple $\bG$-representations. The \emph{spectrum} of $\rho\in \Rep(\bG)$ is
\begin{equation*}
  \spc \rho
  :=
  \left\{\alpha\in \Irr(\bG)\ :\ \alpha\le \rho\right\}
  \quad
  \text{(i.e. $\alpha$ is a summand of $\rho$)}.
\end{equation*}

Recall also \cite[Definition 4.21(i)]{hm5} that $\rho\in\Rep(\bG)$ is \emph{isotypic} if it decomposes as a sum of copies of a single irreducible $\bG$-representation.

\begin{theoremN}\label{thn:sel.iff.z0g}
  Let $\bG$ be a compact connected group and $\rho\in \Rep(\bG)$.
  \begin{enumerate}[(1),wide]
  \item\label{item:thn:sel.iff.z0g:ann2sel} If $\rho$ contains an irreducible summand annihilating $Z_0(\bG)\cap \bG'$ then it selects eigenvalues.

  \item\label{item:thn:sel.iff.z0g:sel2ann} Conversely, if $\rho$ has finite spectrum and isotypic restriction $\rho|_{Z_0(\bG)}$, then it can only select eigenvalues if it annihilates $Z_0(\bG)\cap \bG'$. 
  \end{enumerate}  
\end{theoremN}

As an offshoot of the discussion, relating back to the above-mentioned description in \cite[Theorem]{MR210096} of the quotient \Cref{eq:quot.is.bdl} as a disk bundle over a circle, we revisit the matter in \Cref{th:conj.clss.fib}. Paraphrased and compressed here for brevity, that result reads as follows. 

\begin{theoremN}
  \begin{enumerate}[(1),wide]
  \item The space of adjoint orbits of a compact connected group $\bG$ with center $Z(\bG)$ is a fiber bundle over a torus of dimension $\dim Z(\bG)$, with contractible fibers.

  \item The fiber decomposes as a product of $\mathrm{rank}(\bG')$ joins $\bD*(\bS/\Gamma)$ for disks $\bD$ and actions of finite cyclic groups $\Gamma$ on spheres $\bS$.

  \item Said fibers are in fact homeomorphic to disks if the (semisimple) derived subgroup $\bG'\le \bG$ fits into an exact sequence
    \begin{equation*}
      \{1\}
      \to
      \bL      
      \lhook\joinrel\xrightarrow{\quad}
      \bG'
      \xrightarrowdbl{\quad}
      \bG'_{B}
      \to
      \{1\}      
    \end{equation*}
    with $\bL$ simply-connected and $\bG'_B$ having only type-$B$ quotients. 
  \end{enumerate}
\end{theoremN}

\subsection*{Acknowledgements}

I am grateful for valuable input from A. Garnier, I. Gogi{\'c}, Y. Li, D. Sage, A. Sikora and M. Toma\v{s}evi\'c. 


\section{Continuous eigenvalue selectors}\label{se:eigen.sel}

The proof of \Cref{thn:sel.iff.z0g} requires some preparation and auxiliary results. We first record an immediate consequence:

\begin{corollary}\label{cor:irr.sel.ann}
  For any compact connected group $\bG$, $\rho\in \Irr(\bG)$ selects eigenvalues if and only if it annihilates $Z_0(\bG)\cap \bG'$. 
\end{corollary}
\begin{proof}
  \Cref{thn:sel.iff.z0g} applies: part \Cref{item:thn:sel.iff.z0g:ann2sel} obviously, and \Cref{item:thn:sel.iff.z0g:sel2ann} because $\rho\in \Irr(\bG)$ have finite spectrum by definition, and are isotypic on the center $Z(\bG)$ by \emph{Schur's lemma} \cite[Lemma 2.30]{hm5}. 
\end{proof}

Recall \cite[Definitions 9.30]{hm5} that \emph{pro-tori} are by definition compact connected abelian groups, and much of the standard Lie theory of maximal tori replicates for maximal pro-tori \cite[Theorem 9.32]{hm5} of compact connected groups. For a pro-torus $\bT\le \bG$ and $\rho\in \Rep(\bG)$ a \emph{weight} \cite[Definition II.8.2]{btd_lie_1995} of $\rho$ with respect to $\bT$ is a character
\begin{equation*}
  \chi
  \in
  \widehat{\bT}
  :=  
  \text{the \emph{Pontryagin dual} \cite[\S 1.2.1]{rud_lc} of $\bT$}
\end{equation*}
appearing as a summand of the restriction $\rho|_{\bT}$. 

\begin{remark}\label{re:wt.ann.rep.ann}
  In reference to \Cref{thn:sel.iff.z0g}, note that the irreducible representations annihilating $Z_0(\bG)\cap \bG'$ are precisely those whose underlying weights annihilate the same group.

  Indeed, for $\rho\in \Irr(\bG)$, one $\rho$-weight annihilates $\bH\le Z(\bG)$ if and only if all $\rho$-weights do: irreducible representations are isotypic on $Z(\bG)$ by Schur \cite[Lemma 2.30]{hm5}, and every pro-torus contains $Z(\bG)$ \cite[Theorem 9.32(iv)]{hm5}; $\psi'\cdot\psi^{-1}$ thus annihilates $Z(\bG)$ for any two $\rho$-weights attached to a maximal pro-torus $\bT\le \bG$.

  In short: for $\rho\in\Rep(\bG)$, the $\rho$-weights annihilating $\bH\le Z(\bG)$ are precisely those of the irreducible $\rho$-summands annihilating $\bH$.
\end{remark}

The following simple observation will occasionally help reduce the problem of assessing whether or not a continuous selector exists to simpler groups. We denote \emph{point spectra} \cite[Exercise VII.5.1]{ds_linop-1_1958} of operators by $\sigma_p$.

\begin{lemma}\label{le:cont.same.sp}
  Let $\bG$ be a compact connected group, $\rho\in \Rep(\bG)$ and $\bG\xrightarrow{\varphi}\bS^1$ a continuous $\rho$-selector.

  If $\Psi=\left\{\psi\right\}\subset \cat{Cont}(\bG\rightarrow \bG)$ is a family connected in the uniform topology with $\sigma_p\circ \rho$ constant along each $\Psi$-orbit $\Psi x:=\left\{\psi x\ :\ \psi\in \Psi\right\}$ then $\varphi$ is constant on every $\Psi$-orbit. 
\end{lemma}
\begin{proof}
  The value $\varphi(x)$, $x\in \bG$ belongs to $\sigma_p \left(\rho(x)\right)$, a set we are assuming constant as $x\in \bG$ ranges any individual $\Psi$-orbit. It is enough to prove that $\varphi|_{\Psi x}$ is constant for a dense set of $x$. Given the connectedness of $\Psi$ (hence also that of $\Psi x$), it furthermore suffices to argue that
  \begin{equation*}
    \left\{x\in \bG\ :\ \sigma_p\left(\rho(x)\right)\text{ is totally disconnected}\right\}
    \subseteq \bG
    \quad\text{is dense}
  \end{equation*}
  (\emph{total disconnectedness} meaning \cite[Definition 29.1]{wil_top} that points are the largest connected subspaces). Note furthermore that $\sigma_p\left(\rho(x)\right)$ is contained in
  \begin{equation*}
    \widehat{\bT}|_x
    :=
    \left\{\chi(x)\ :\ \chi\in\widehat{\bT}\right\}
  \end{equation*}
  for any maximal pro-torus $\bT\ni x$ of $\bG$ (which pro-torus is chosen makes no difference, given their mutual conjugacy \cite[Theorem 9.32(i)]{hm5}), so it suffices to prove the stronger claim that
  \begin{equation}\label{eq:x.w.totdisc.vset}
    \left\{x\in \bT\ :\ \widehat{\bT}|_{x}\subseteq \bS^1\text{ totally disconnected}\right\}
    \subseteq
    \bT
    \quad\text{is dense}.
  \end{equation}
  To that end, substitute $\bS:=\overline{\braket{x}}$ for $\bT$. Nothing is lost in passing, moreover, to compact groups surjecting onto $\bS$, so we can assume the latter \emph{free} compact abelian on the single generator $x$:
  \begin{equation}\label{eq:ss1d}
    \bS
    \overset{\text{\cite[Theorem 8.53]{hm5}}}{\cong}
    \widehat{\bS^1_d}
    ,\quad
    \bS^1_d:=\text{discrete circle group}.
  \end{equation}
  The standard classification \cite[Theorem 4]{kap} of \emph{divisible} abelian groups gives
  \begin{equation*}
    \bS\cong \widehat{\bQ}^{2^{\aleph_0}}\times \prod_{\text{primes }p}\bZ_p
    ,\quad
    \left(\bZ_p,+\right)
    :=
    \text{\emph{$p$-adic integers} \cite[Example 2.1.6(2)]{rz_prof_2e_2010}},
  \end{equation*}
  and every element therein is arbitrarily approximable by elements of the corresponding direct \emph{sum}
  \begin{equation*}
    \widehat{\bQ}^{\oplus 2^{\aleph_0}}\oplus \bigoplus_{\text{primes }p}\bZ_p    
  \end{equation*}
  (as opposed to the direct product). That direct sum is contained in the left-hand set of \Cref{eq:x.w.totdisc.vset}, concluding the proof. 
\end{proof}

As a simple consequence of \Cref{le:cont.same.sp}, continuous selectors descend to quotients by connected kernels.

\begin{corollary}\label{cor:qout.conn.ker}
  If a representation $\rho\in \Rep(\bG)$ of a compact connected group factors through a quotient $\bG/\bK$ by a connected closed subgroup $\bK\trianglelefteq \bG$ then so does every continuous $\rho$-selector. 
\end{corollary}
\begin{proof}
  Apply \Cref{le:cont.same.sp} with the family of right translates by $k\in \bK$ as $\Psi$. 
\end{proof}

\begin{remark}\label{re:not.always.const.cosets}
  The connectedness of $\bK$ is crucial in \Cref{cor:qout.conn.ker}: the representation
  \begin{equation*}
    \chi_{2m}\oplus \chi_{2n}
    ,\quad
    m\ne n\in \bZ
    ,\quad
    \chi_d:=\left(z\xmapsto{\quad}z^d\right)
  \end{equation*}
  of $\bG:=\bS^1$ factors through $\bS^1/\left(\bZ/2\right)$, but has a continuous selector
  \begin{equation*}
    z
    \xmapsto{\quad\varphi\quad}
    \begin{cases}
      z^{2m}&z\in \exp \pi i \left[0,1\right]\\
      z^{2n}&z\in \exp \pi i \left[-1,0\right]\\
    \end{cases}
  \end{equation*}
  which does not (plainly, $\varphi(z)\ne \varphi(-z)$ in general).
\end{remark}

A few more consequences of the preceding discussion follow. 

\begin{corollary}\label{cor:prod.sel}
  \begin{enumerate}[(1),wide]
  \item\label{item:cor:prod.sel:quot.by.conn} If a compact connected group $\bG$ selects eigenvalues, so does any quotient $\bG/\bK$ by a connected closed subgroup $\bK\trianglelefteq \bG$.

  \item\label{item:cor:prod.sel:prod} A product $\prod_{i\in I}\bG_i$ of compact connected groups selects eigenvalues if and only if each individual factor does. 
  \end{enumerate}
\end{corollary}
\begin{proof}
  Part \Cref{item:cor:prod.sel:quot.by.conn}, at this stage, is immediate: a $\bG/\bK$ representation has a continuous selector over $\bG$, factoring through $\bG/\bK$ by \Cref{cor:qout.conn.ker}.

  The forward implication of \Cref{item:cor:prod.sel:prod} follows from part \Cref{item:cor:prod.sel:quot.by.conn}; as to ($\Leftarrow$), an irreducible representation of $\prod \bG_i$ factors through some finite product and is thus \cite[Theorem 3.2]{mack-unit} a finite external tensor product
  \begin{equation*}
    \bigotimes_{i\in \text{finite set }F}\rho_i
    ,\quad
    \rho_i\in\Irr(\bG_i)
    ,\quad
    i\in F. 
  \end{equation*}
  Now simply take for our sought-after continuous selector the map
  \begin{equation*}
    \prod_{i\in I}\bG_i
    \xrightarrowdbl{\quad}
    \prod_{i\in F}\bG_i
    \xrightarrow{\quad\prod_{i\in F}\varphi_i\quad}
    \bS_1
  \end{equation*}
  for continuous respective $\rho_i$-selectors $\varphi_i$. 
\end{proof}

One type of reduction alluded to before the statement of \Cref{le:cont.same.sp} is as follows. 

\begin{corollary}\label{cor:fact.lie}
  If a representation $\rho\in \Rep(\bG)$ of a compact connected group factors through a Lie quotient of $\bG$, so do its continuous selectors. 
\end{corollary}
\begin{proof}
  If $\rho$ is in fact a representation of the Lie quotient $\bG/\bK$, \Cref{cor:qout.conn.ker} shows that all of its continuous selectors will factor through $\bG/\bK_0$. It remains to observe that the latter quotient is also Lie if $\bG/\bK$ is. The general framework is as follows:
  \begin{itemize}[wide]
  \item a compact connected group $\bL:=\bG/\bK_0$;

  \item a Lie quotient $\bL/\bD$ thereof, with $\bD:=\bK/\bK_0$ totally disconnected (and hence also central, for it is normal in the connected $\bL$);

  \item the claim being that $\bL$ must then have been Lie to begin with; this follows from \cite[Lemma 4.4]{iw}.
  \end{itemize}
\end{proof}

The following result is simpler under additional conditions (\Cref{re:easier.if.sml}), but perhaps not immediately obvious in the full generality in which it is stated here.

\begin{proposition}\label{pr:const.on.conj.cls}
  Let $\bG$ be a compact connected group and $\rho\in \Rep(\bG)$. Every continuous $\rho$-selector $\bG\xrightarrow{\varphi}\bS^1$ is constant on $\bG$-conjugacy classes. 
\end{proposition}
\begin{proof}
  This is an application of \Cref{le:cont.same.sp} to
  \begin{equation*}
    \Psi:=\left\{\Ad_g\ :\ g\in \bG\right\},
  \end{equation*}
  where

  \begin{equation*}
    (\bG,\ \bG)
    \ni
    (g,x)
    \xmapsto{\quad\Ad=\Ad_{\bG}\quad}
    :=
    \Ad_g(x)
    :=
    gxg^{-1}
  \end{equation*}
  denotes the \emph{adjoint action} of $\bG$.
\end{proof}

\begin{remark}\label{re:easier.if.sml}
  \Cref{pr:const.on.conj.cls} is easier under fairly reasonable smallness constraints on $\bG$: if $\bT$ is metrizable (so in particular if $\bG$ is Lie) then $\widehat{\bT}$ is countable and \Cref{eq:x.w.totdisc.vset} is in fact an equality. This is not so in general: $\bT$ might be a pro-torus containing a free compact abelian group on a single generator $x$, whereupon the identification \Cref{eq:ss1d} shows that $\widehat{\bT}|_{x}$ is all of $\bS^1$. 
\end{remark}


In the following statement, compact connected groups are
\begin{itemize}[wide]
\item \emph{semisimple} \cite[Definition 9.5]{hm5} if equal to their own commutators (equivalently \cite[Theorem 9.24(i)]{hm5}, compact connected $\bG$ with trivial central identity component $Z_0(\bG)$);

\item and \emph{adjoint} \cite[p.30]{stein_chev} if center-less. Equivalently \cite[Corollary 9.25]{hm5}, products of simple connected compact adjoint Lie groups. 
\end{itemize}

\begin{proposition}\label{pr:prod.protor.ss}
  Compact connected groups splitting as products of the form
  \begin{equation*}
    \bG
    \cong
    \left(\text{pro-torus}\right)
    \times
    \left(\text{simply-connected factor}\right)
    \times
    \left(\text{adjoint factor}\right)
  \end{equation*}
  select eigenvalues. 
\end{proposition}
\begin{proof}
  \Cref{cor:prod.sel}\Cref{item:cor:prod.sel:prod} reduces the problem to handling the three individual factors, but we will aggregate the first two.

  \begin{enumerate}[(I),wide]
  \item\label{item:pr:prod.protor.ss:sc} \textbf{: Pro-torus-simply-connected products.} We can always restrict attention to finite-dimensional representations, which will factor through Lie quotients of $\bG$. It is thus enough \cite[Theorem 9.24]{hm5} to assume that $\bG$ is of the form
    \begin{equation*}
      \bG\cong \bA\times \prod_{i=1}^n \bS_i
      ,\quad
      \bA\text{ a pro-torus and }
      \bS_i
      \text{ simple, connected, simply-connected}.
    \end{equation*}  
    A continuous eigenvalue selector would in any case be constant on conjugacy classes of $\bG$ (\Cref{pr:const.on.conj.cls}), so we seek to define that selector on $\bG/\Ad_{\bG}$. Denote by
    \begin{equation*}
      \bT:=\bA\times \prod_i \bT_i
      ,\quad
      \bT_i\le \bS_i\text{ maximal tori}
    \end{equation*}
    a maximal torus of $\bG$. The \emph{Weyl group} \cite[Definition IV.1.3]{btd_lie_1995}
    \begin{equation*}
      W_{\bG}
      :=
      N_{\bG}(\bT)/\bT
      ,\quad
      N_{\bG}(\bT)
      :=
      \text{normalizer of $\bT$ in $\bG$}
    \end{equation*}
    is nothing but $\prod_i W_{\bS_i}$, and the description of a \emph{fundamental domain} for Weyl actions on maximal tori of compact, connected, simply-connected simple Lie groups in \cite[\S 4.8, Theorem]{hmph_cox} makes it clear that the surjection $\bT\xrightarrowdbl{\pi}\bT/W_{\bG}$ splits continuously (i.e. admits a continuous right inverse $\iota$). We thus define the sought-after selector as
    \begin{equation*}
      \bG/\Ad_{\bG}
      \overset{\text{\cite[Proposition IV.2.6]{btd_lie_1995}}}{\cong}
      \bT/W_{\bG}    
      \lhook\joinrel\xrightarrow{\quad}
      \bT
      \xrightarrow{\quad\chi\quad}
      \bS^1
    \end{equation*}
    for \emph{any} weight $\chi$ of $\rho$.

  \item \textbf{: Adjoint groups.} This is the simpler of the two cases: in the adjoint case the constant map
    \begin{equation*}
      \bG\ni g
      \xmapsto{\quad}
      1
      \in \bS^1
    \end{equation*}
    is an eigenvalue selector. Having again reduced the problem irreducible representations of compact adjoint Lie groups $\bG$, the \emph{highest weights} featuring in the standard classification \cite[Theorem 14.18]{fh_rep-th} of $\Irr(\bG)$ are precisely
    \begin{equation*}
      \lambda\in\left\{\text{dominant weights in the \emph{root lattice} \cite[post Observation 12.6]{fh_rep-th}}\right\}.
    \end{equation*}
    The corresponding irreducible representations $V_{\lambda}$ all have the trivial character of a maximal torus $\bT\le \bG$ as a weight; this can be seen any number of ways, e.g. as a consequence of \cite[\S 21.3, Proposition]{hmph_1980}, or of \cite[Theorem 14.18]{fh_rep-th} after noting that for dominant weights $\lambda$, regarded as elements of $\ft^*:=Lie(\bT)^*$ the convex hull
    \begin{equation*}
      \mathrm{conv}
      \left\{w\lambda\ :\ w\in \text{ Weyl group}\right\}
    \end{equation*}
    contains the origin.
  \end{enumerate}
\end{proof}

The qualifications on the semisimple factor in \Cref{pr:prod.protor.ss} turn out not to be necessary.

\begin{proposition}\label{pr:prod.protor.ss.better}
  Products of pro-tori and semisimple compact connected groups select eigenvalues.
\end{proposition}
\begin{proof}
  Another application of \Cref{cor:prod.sel}\Cref{item:cor:prod.sel:prod} disposes of the pro-torus factor, so fix
  \begin{equation*}
    \left(
      \bG
      \xrightarrow{\quad\rho\quad}
      \U(V)
    \right)
    \in\Irr(\bG)
    ,\quad
    V\text{ a finite-dimensional Hilbert space}
  \end{equation*}
  for semisimple (compact connected) $\bG$, which may as well be assumed Lie (for $\rho$ factors through a Lie quotient). Now simply observe that $\rho$ must in fact factor through $\SU(V)\subset \U(V)$ because $\bG$ is semisimple, so we can restrict an eigenvalue selector for the defining representation of $\SU(V)$ on $V$ (one exists by \Cref{pr:prod.protor.ss}) along $\bG\xrightarrow{\rho}\SU(V)$. 
\end{proof}

\Cref{pr:irr.ann.z0g} is a partial converse to \Cref{pr:prod.protor.ss.better}.

\begin{proposition}\label{pr:irr.ann.z0g}
  Let $\bG$ be a compact connected group and $\rho\in \Rep(\bG)$ an finite-spectrum representation, isotypic on $Z_0:=Z_0(\bG)$. 

  If $\rho$ selects eigenvalues then it annihilates $Z_0\cap \bG'$.
\end{proposition}
\begin{proof}
  We will argue by contradiction that $\rho$ annihilates any one fixed but arbitrary element $x\in Z_0(\bG)\cap \bG'$. The finite-spectrum assumption implies that $\rho$ factors through a Lie quotient $\bG\xrightarrowdbl{}\overline{\bG}$; that quotient map restricts to morphisms
  \begin{equation*}
    \bG'\xrightarrowdbl[\quad\text{onto}\quad]{}\overline{\bG}'
    \quad\text{and}\quad
    Z(\bG)\xrightarrowdbl[\quad\text{onto}\quad]{}Z\left(\overline{\bG}\right)
  \end{equation*}
  by \cite[Proposition 9.26]{hm5}, so the fixed element $x$ will map into $Z_0(\overline{\bG})\cap \overline{\bG}'$. We can thus substitute $\overline{\bG}$ for $\bG$ and assume the latter Lie to begin with.
 
  By assumption, $\rho$ restricts to $Z_0:=Z_0(\bG)$ as a sum of copies of a single character $\theta\in \widehat{Z_0}$. I first claim that we have  
  \begin{equation}\label{eq:phi.mod.prop}
    \varphi(sg) = \theta(s)\varphi(g)
    ,\quad
    \forall \left(s\in Z_0,\ g\in \bG\right).
  \end{equation}
  Assume this for the moment. There is no harm in further quotienting out the kernel of $\theta|_{Z_0\cap \bG'}$, so that in particular $\theta$ is faithful on the finite cyclic group $\Gamma:=\braket{x}$. Restricted to $\bG'$, \Cref{eq:phi.mod.prop} gives a continuous map $\bG'\to \bS^1$, equivariant for two free actions by $\Gamma$: translation by $x$ on $\bG'$ and translation by $\theta(x)$ on $\bS^1$.
  
  \cite[Remark on p.68]{MR1247532} shows that such a continuous equivariant map $\bG'\xrightarrow{\varphi|_{\bG'}}\bS^1$ cannot be nullhomotopic. On the other hand, it must be: $\bG'$ (being a compact connected semisimple Lie group) has finite fundamental group \cite[Remark V.7.13]{btd_lie_1995}, the map $\varphi|_{\bG'}$ thus annihilates fundamental groups, so lifts to the contractible universal cover $\bR\xrightarrowdbl{}\bS^1$ by \cite[Lemma 79.1]{mnk}.

  It remains to prove \Cref{eq:phi.mod.prop}. To that end, fix $g\in \bG$ and note that
  \begin{equation*}
    Z_{0\mid \lambda}
    :=
    \left\{s\in Z_0\ :\ \varphi(sg)=\theta(s)\lambda\right\}
    ,\quad
    \lambda\in\sigma_p(\rho(g))
  \end{equation*}
  constitute a finite clopen partition of the connected space $Z_0$, so exactly one $Z_{0\mid \lambda}$ must be all of $Z_0$ (and the rest empty). Since on the other hand $Z_0\ni 1$ belongs to $Z_{0\mid \varphi(g)}$, the latter must be non-empty and we are done.
\end{proof}

\begin{remark}\label{re:isotyp.nec}
  The isotypic condition in \Cref{pr:irr.ann.z0g} is not redundant: \emph{every} representation $\rho$ of an arbitrary compact connected $\bG$ embeds as a summand of a representation admitting continuous eigenvalue selectors.

  To see this, assume $\rho\in \Rep(\bG)$ finite-dimensional on the carrier Hilbert space $V$ and consider the direct sum
  \begin{equation*}
    \bG
    \xrightarrow{\quad\rho\oplus \overline{\rho}\quad}
    \SU\left(V^{\oplus 2}\right)
    ,\quad
    \overline{\bullet}:=\text{complex conjugation}.
  \end{equation*}
  That direct sum selects eigenvalues over $\bG$ because (by \Cref{pr:prod.protor.ss.better}) it does over the semisimple (simply-connected \cite[Theorem 1.24(1)]{sep_cpct-lie}) compact group $\SU\left(V^{\oplus 2}\right)$. One can easily arrange for $\rho$ \emph{not} to annihilate $Z_0(\bG)\cap \bG'$.
\end{remark}

\pf{thn:sel.iff.z0g}
\begin{thn:sel.iff.z0g}
  Part \Cref{item:thn:sel.iff.z0g:sel2ann} is precisely \Cref{pr:irr.ann.z0g}, so we focus on \Cref{item:thn:sel.iff.z0g:ann2sel}. Assuming $\rho$ irreducible with $\ker\rho\ge Z_0(\bG)\cap \bG'$, we can substitute for $\bG$ its quotient \cite[Theorem 9.24(i)]{hm5}
  \begin{equation*}
    \bG/\left(Z_0(\bG)\cap \bG'\right)
    \cong
    \bA\times \bG'
    ,\quad
    \bA\text{ a pro-torus};
  \end{equation*}
  the conclusion follows from \Cref{pr:prod.protor.ss.better}.
\end{thn:sel.iff.z0g}

The continuous selectors constructed in the proof of \Cref{pr:prod.protor.ss} only employ a single weight of the representation. The following notion formalizes that type of construction. 

\begin{definition}\label{def:coh.sel}
  A continuous eigenvalue selector for a representation $\rho\in \Rep(\bG)$ of a compact connected group is \emph{coherent} if it is of the form
  \begin{equation}\label{eq:coh.sel}
    \bG\ni g
    \xmapsto{\quad}
    \chi(g)
    \in \bS^1
  \end{equation}
  for a weight $\chi\in \widehat{\bT}$ of $\rho$ with respect to a maximal torus $\bT$, \Cref{eq:coh.sel} being valid for any maximal pro-torus $\bT'\ni x$ upon identifying $\bT'$ and $\bT$ via some inner automorphism of $\bG$.

  We refer to $\chi$ as a \emph{representing weight} of the coherent selector \Cref{eq:coh.sel}. 
\end{definition}

\begin{remarks}\label{res:how.coh.makes.sense}
  \begin{enumerate}[(1),wide]
  \item The very definition of coherence relies implicitly on the mutual conjugacy of all maximal pro-tori, together with the fact that they cover $\bG$ \cite[Theorem 9.32, (i) and (ii)]{hm5} that maximal pro-tori are mutual conjugates. Note that ``the'' representing weight is (at most) defined only up to the Weyl-group action on the discrete Pontryagin dual $\widehat{\bT}$.

  \item It is not difficult to produce incoherent eigenvalue selectors, for instance for non-trivial pro-tori: all such surject onto $\bS^1$ so it is enough to handle the latter. Identifying
    \begin{equation*}
      \bZ
      \ni n
      \xmapsto[\quad\cong\quad]{}
      \chi_n:=
      \left(z\xmapsto{\quad} z^n\right)
      \in
      \widehat{\bS^1},    
    \end{equation*}
    observe that every 2-dimensional $\bS^1$-representation of the form $\chi_m\oplus \chi_n$, $|m-n|>1$ has an incoherent continuous selector: fix two distinct $(m-n)^{th}$ roots of unity $z_i\in \bS^1$, $i=1,2$ and define the selector $\varphi$ as $z^n$ and $z^m$ respectively on the two arcs connecting $z_i$.
  \end{enumerate}
\end{remarks}

We now have the following consequence of \Cref{pr:irr.ann.z0g}. 

\begin{corollary}\label{cor:coh.triv.z0}
  Let $\bT\le\bG$ be a maximal pro-torus of a compact connected group and $\rho\in \Rep(\bG)$.

  If $\chi\in\widehat{\bT}$ is a weight of $\rho$ representing a continuous $\rho$-selector $\bG\xrightarrow{\varphi}\bS^1$ then $\chi$ annihilates $Z_0(\bG)\cap \bG'$. 
\end{corollary}
\begin{proof}
  Simply substitute for $\rho$ an irreducible summand thereof which still has $\chi$ as a weight and apply \Cref{pr:irr.ann.z0g}. 
\end{proof}

\section{Fibered conjugacy-class spaces}\label{se:fib.ad}

Recall (e.g. \cite[Theorem]{MR210096}, slightly paraphrased) that the space of conjugacy classes of the unitary $n\times n$ group is of the form
\begin{equation*}
  \U(n)/\Ad
  \cong
  \text{a $\bD^{n-1}$-bundle over $\bS^1$}
  \quad
  \left(\text{trivial for odd $n$, non-orientable otherwise}\right),
\end{equation*}
with $\bD^d$ denoting the $d$-dimensional disk. That result instantiates the more general compact-Lie-group phenomenon recorded in \Cref{th:conj.clss.fib}. We refer to reader to \cite[Chapter 2]{hus_fib}, \cite[Chapter 1]{hjjm_bdle}, \cite[Chapters 3 and 14]{td_alg-top} and any number of other sources cited therein for basic background on (locally trivial) \emph{bundles} (or \emph{fibrations}), and \cite[Chapter 0]{hatch_at} or \cite[Chapter 2 and Appendix A]{rs_pl_1982} (say) for instance, for brief reminders on \emph{CW complexes} \cite[\S A.7]{rs_pl_1982} and various operations pertinent to piecewise-linear topology (e.g. the \emph{joins} $X*Y$ of \cite[pp.9-10]{hatch_at}).

\begin{theorem}\label{th:conj.clss.fib}
  Let $\bG$ be a compact connected Lie group.

  \begin{enumerate}[(1),wide]
  \item\label{item:th:conj.clss.fib:gen} The space $\bG/\Ad$ of adjoint orbits of $\bG$ is a bundle over the abelianization
    \begin{equation*}
      \bG_{ab}:=\bG/\bG'\cong \bT^{\dim Z(\bG)}
    \end{equation*}
    with fiber
    \begin{equation}\label{eq:ddidsi}
      \prod_{i=1}^k \left(\bD^{d_i}*\left(\bS^{s_i}/\Gamma_i\right)\right)
      ,\quad
      \sum_i \left(d_i+s_i+1\right)=\rk(\bG')
    \end{equation}
    where
    \begin{itemize}[wide]
    \item $\rk$ denotes the \emph{rank} \cite[Definition IV.2.1]{btd_lie_1995} of a compact connected Lie group;

    \item all $d_i$ are non-negative (so that the factors in \Cref{eq:ddidsi} are all contractible);

    \item the number $k$ of factors is that of simple ideals of the semisimple Lie algebra $\fg':=Lie(\bG')$;
      
    \item and $\Gamma_i=\braket{\omega}$ are finite cyclic groups acting faithfully on spheres $\bS^{s_i}$, with
      \begin{equation*}
        \mathrm{fix}(\omega)
        :=
        \left\{x\in \bS^{s_i}\ :\ \omega(x)=x\right\}=\emptyset.
      \end{equation*}
    \end{itemize}
    
  \item\label{item:th:conj.clss.fib:bsc} Consider the decomposition
    \begin{equation*}
      \fg'=\fb\oplus \fl
      ,\quad
      \fb:=\bigoplus_{\substack{\text{simple }\fs\trianglelefteq \fg'\\\text{$\fs$ of type $B$}}}
      \fs.
    \end{equation*}
    If the connected subgroup $\bL\le \bG'$ with $Lie(\bL)=\fl$ is simply-connected, then the fiber \Cref{eq:ddidsi} is homeomorphic to a ball $\bD^{\rk(\bG')}$.    
  \end{enumerate}
\end{theorem}
\begin{proof}
  There is in any case an identification \cite[Proposition IV.2.6]{btd_lie_1995} $\bG/\Ad\cong \bT/W_{\bG}$ valid for any compact connected Lie group, $W_{\bG}:=N_{\bG}(\bT)/\bT$ denoting the Weyl group attached to a maximal torus $\bT\le \bG$, which we use take for granted throughout.

  Consider the decomposition \cite[Theorem 6.4]{hm5}
  \begin{equation*}
    \fg:=Lie(\bG)
    =
    \fz\oplus \fg'
    =
    \fz\oplus \bigoplus_i \fs_i
    ,\quad
    \fz:=\text{center $Z(\fg)$}
    ,\quad
    \fs_i\text{ simple}. 
  \end{equation*}
  Each ideal $\fl\trianglelefteq \fg$ is the Lie algebra of a unique \cite[\S III.6.2, Theorem 2]{bourb_lie_1-3} connected subgroup $\bL\le \bG$, automatically closed  (e.g. by \cite[Chapter II, Exercise D.4]{helg_dglgssp}). Picking maximal tori $\bT_{\bL}\le \bL$ and $\bT_{\bG/\bL}\le \bG/\bL$ exhibits the quotient $\bT_{\bG}/W_{\bG}$ of the unique maximal torus
  \begin{equation*}
    \bT_{\bG}\le \bG
    \quad\text{with}\quad
    \bT_{\bL}\le \bT_{\bG}
    \quad\text{and}\quad
    \bT_{\bG}\text{ surjecting onto }\bT_{\bG/\bL}
  \end{equation*}
  as a locally trivial $\bT_{\bL}/W_{\bL}$-bundle over $\bT_{\bG/\bL}/W_{\bG/\bL}$. Applying this to $\fl:=\fz$ first reduces the problem to showing that the fiber $\bT_{\bG'}/W_{\bG'}$ of the bundle
  \begin{equation*}
    \begin{tikzpicture}[>=stealth,auto,baseline=(current  bounding  box.center)]
      \path[anchor=base] 
      (0,0) node (l) {$\bT_{\bG'}/W_{\bG'}$}
      +(3,.5) node (u) {$\bT_{\bG}/W_{\bG}$}
      +(6,-.5) node (r) {$\bT_{\bG_{ab}}/W_{\bG_{ab}}$}
      ;
      \draw[right hook->] (l) to[bend left=6] node[pos=.5,auto] {$\scriptstyle $} (u);
      \draw[->>] (u) to[bend left=6] node[pos=.5,auto,swap] {$\scriptstyle $} (r);
    \end{tikzpicture}
  \end{equation*}
  is (homeomorphic to) a ball. In other words, it will suffice (by switching focus from $\bG$ to $\bG'$) to show that $\bG/\Ad$ is (homeomorphic to) a ball of dimension $\rk(\bG)$ for compact, connected \emph{semisimple} Lie groups $\bG$.

  The same fibration principle also reduces the problem to simple $\bG$, by selecting $\fl:=\fs_i$ for some $i$: locally trivial bundles over a contractible spaces (sufficiently well-behaved, e.g. \emph{CW complexes} by \cite[Theorem 1.3.5]{fp_cell}) are trivial by standard bundle-classification theory \cite[Theorems 14.4.1 and 14.4.2]{td_alg-top}. The fibration 
  \begin{equation*}
    \begin{tikzpicture}[>=stealth,auto,baseline=(current  bounding  box.center)]
      \path[anchor=base] 
      (0,0) node (l) {$\bT_{\bS_i}/W_{\bS_i}$}
      +(3,.5) node (u) {$\bT_{\bG}/W_{\bG}$}
      +(6,-.5) node (r) {$\bT_{\bG/\bS_i}/W_{\bG/\bS_i}$}
      ;
      \draw[right hook->] (l) to[bend left=6] node[pos=.5,auto] {$\scriptstyle $} (u);
      \draw[->>] (u) to[bend left=6] node[pos=.5,auto,swap] {$\scriptstyle $} (r);
    \end{tikzpicture}
  \end{equation*}
  thus splits as
  \begin{equation*}
    \bT_{\bG}/W_{\bG}
    \cong
    \bT_{\bS_i}/W_{\bS_i}
    \times
    \bT_{\bG/\bS_i}/W_{\bG/\bS_i},
  \end{equation*}
  so we can further assume (by induction on the number of simple ideals $\fs_i\trianglelefteq \fg$) $\bG$ simple throughout the rest of the proof. The goal is thus to argue that  
  \begin{equation*}
    \bG\text{ compact connected simple}
    \quad
    \xRightarrow{\quad}
    \quad
    \bG/\Ad\cong
    \bD^d*\left(\bS^s/\Gamma\right)
    ,\quad d\ge 1
  \end{equation*}
  for some free action of a finite abelian group $\Gamma$ on a sphere $\bS^s$ (by necessity, $d+s+1=\rk(\bG)$). 
  
  \begin{enumerate}[(I),wide]
  \item\label{item:th:conj.clss.fib:sc} \textbf{: Simply-connected $\bG$.} In that case, \cite[\S 4.8, Theorem]{hmph_cox} (effectively) identifies $\bT/W$ with a simplex. In particular, this settles those cases where the group $\bG$ (assumed simple) is of one of the automatically-simply-connected types \cite[table on p.30, following Lemma 28]{stein_chev}: $E_8$, $F_4$ and $G_2$. 

  \item\label{item:th:conj.clss.fib:gen} \textbf{: Generalities on arbitrary (simple) $\bG$.} Denote by $\widetilde{\bG}$ the \emph{universal cover} \cite[pp.54-55, post Theorem 12]{stein_chev} of the connected simple compact Lie group $\bG$. The sought-after space $\bG/\Ad\cong \bT/W$ is then a quotient
    \begin{equation*}
      \left(\widetilde{\bT}/W\cong \widetilde{\bG}/\Ad\right)/\pi_1(\bG)
    \end{equation*}
    of the simplex
    \begin{equation*}
      \Delta^n:=\widetilde{\bT}/W\cong \widetilde{\bG}/\Ad
      ,\quad
      n:=\rk(\bG)
    \end{equation*}
    of \Cref{item:th:conj.clss.fib:sc} above, with $\pi_1(\bG)\le \widetilde{\bG}$ embedded as a finite central subgroup. The possible fundamental groups $\pi_1(\bG)$ are the subgroups of $Z(\widetilde{\bG})$ listed in \cite[2$^{nd}$ column in table on p.30]{stein_chev}, and the action on the $\rk(\bG)$-simplex $\widetilde{\bT}/W$ is described in \cite[Table 3, pp.64-65]{garnier:tel-03622954}. 
    
  \item \textbf{: Type $B_{\ge 2}$.} In that case the only non-simply-connected variant is the adjoint group, with $\pi_1\cong \bZ/2$. The type-$B$ rows in \cite[Table 3]{garnier:tel-03622954} show that the non-trivial involution interchanges two vertices of the simplex and fixes the rest; plainly, the quotient is topologically a ball. 

    The argument thus far settles \Cref{item:th:conj.clss.fib:gen} partially and \Cref{item:th:conj.clss.fib:bsc} completely.

  \item\label{item:th:conj.clss.fib:cycl} \textbf{: Cyclic $\pi_1(\bG)$.} The group $\pi_1:=\pi_1(\bG)$ is cyclic generated by $\omega$ operating on the $n+1$ vertices of $\Delta^n$ as a product
    \begin{equation*}
      w=\sigma_1\cdots \sigma_k
      ,\quad
      1\le k\le n
      ,\quad
      \sigma_i\text{ disjoint cycles of respective lengths }1\le \ell_i\le n+1.
    \end{equation*}
    Each $\sigma_i$ respectively cycles through the $k$ vertices of a $(k-1)$-simplex $\Delta_i$. Set
    \begin{equation*}
      \begin{aligned}
        \bD^{k-1}
        &:=
          \Delta^n\cap L^{k-1}
        \text{ with}\\
        L^{k-1}
        &:=
          \aspn \left\{\text{\emph{barycenter} \cite[\S 2]{zbMATH01714503} of }\Delta_i\ :\ 1\le i\le k\right\}\\
        \aspn
        &:=
          \text{\emph{affine span} \cite[Definition pre Theorem 16.4]{rom_lalg_3e_2008}}.
      \end{aligned}      
    \end{equation*}
    The affine automorphism $\omega$ of $\aspn \Delta^n$ fixes $L^{k-1}$ pointwise and operates fixed-point-freely on the \emph{link} \cite[\S 2, p.222]{zbMATH01714503}
    \begin{equation*}
      \lk_{\Delta^n}\left(\bD^{k-1}\right)
      \cong
      \bS^{n-k}
    \end{equation*}
    of $\bD^{k-1}$ in $\Delta^n$. We thus have
    \begin{equation}\label{eq:join.ball.lens}
      \Delta^n/\pi_1
      \cong
      \bD^{k-1}*\left(\lk_{\Delta^n}\left(\bD^{k-1}\right)/\braket{\omega}\right)
      \cong
      \bD^{k-1}*\left(\bS^{n-k}/\braket{\omega}\right),
    \end{equation}
    as desired. 

    In addition to the previously-covered cases, this now also handles
    \begin{itemize}[wide]
    \item all type-$A$ groups (with $|\pi_1|$ ranging through the divisors of $n+1$ for type $A_n$);
    \item all types $C$ with $|\pi_1|=2$, so that the sphere quotients in \Cref{eq:join.ball.lens} are real projective spaces $\bR\bP^{n-k}$;

    \item intermediate (i.e. non-adjoint, non-simply-connected) types $D_{\text{even}}$ with $|\pi_1|=2$ again;

    \item adjoint types $D_{\text{odd}}$ with $|\pi_1|=4$;

    \item and types $E_6$ and $E_7$, with $|\pi_1|=3,2$ respectively. 
    \end{itemize}

  \item\label{item:th:conj.clss.fib:adjd} \textbf{: Adjoint type $D_{\text{even}\ge 4}$.} The fundamental group is $(\bZ/2)^2$, with the non-trivial elements operating on the vertex set of $\Delta^n$ as
      \begin{equation}\label{eq:d2n.omegas}
        \omega_1=\sigma_1
        \quad\text{and}\quad
        \omega_{2,3}=\sigma_{2,3}\cdot\prod_{i=1}^{\frac{n-4}2}\tau_i
        \quad\text{respectively},
      \end{equation}
      with $\sigma_{1,2,3}$ denoting the affine involutions of a simplex $\Delta^3\subset \Delta^n$ fixing no vertices of $\Delta^3$ (so each is a product of two disjoint transpositions) and $\tau_i$ denoting mutually-disjoint transpositions of pairs of vertices of the complementary $\Delta^{n-4}\subset \Delta^n$ (i.e. the link \cite[\S 2, p.222]{zbMATH01714503} on $\Delta^3\subset \Delta^n$).

      Observe first that the quotient $\Delta^n/\braket{\omega_1}$ is an $n$-ball: because $\omega_1$ acts trivially on $\Delta^{n-4}$ and as reflection across a \emph{bimedian} (segment connecting two opposite edges) of $\Delta^3$, we have
    \begin{equation*}
      \Delta^n/\braket{\omega_1}
      \cong
      \left(\Delta^3*\Delta^{n-4}\right)/\braket{\omega_1}
      \cong
      \left(\Delta^3/\braket{\omega_1}\right)*\Delta^{n-4}
      \cong
      \bD^3*\Delta^{n-4}
      \cong
      \bD^n.
    \end{equation*}
    Quotienting further by $\braket{\omega_2}$ effectively reduces us to \Cref{item:th:conj.clss.fib:cycl} above.    
  \end{enumerate}
\end{proof}

\begin{remark}\label{re:whycontractible}
  The contractibility of the individual factors in \Cref{eq:ddidsi} is in any case immediate (regardless of any other structure results) once we identify the simplex $\Delta:=\widetilde{\bG}/\Ad$ with the \emph{cone} \cite[p.225]{zbMATH01714503}
  \begin{equation*}
    \cC \partial\Delta
    :=
    \{\mathrm{pt}\}*\partial \Delta
    ,\quad
    \partial\Delta:=\text{\emph{boundary} \cite[p.222]{zbMATH01714503} of $\Delta$}
  \end{equation*}
  with the barycenter of $\Delta$ as the tip and observe that the $\pi_1(\bG)$-action fixes that tip, so that $\Delta/\pi_1$ is the cone $\cC\left(\partial\Delta/\pi_1\right)$. 
\end{remark}

It is certainly not the case, in the context of \Cref{th:conj.clss.fib}, that the fibers \Cref{eq:ddidsi} are always homeomorphic to balls: step \Cref{item:th:conj.clss.fib:cycl} of the above proof, for instance, gives enough information to determine $\bG/\Ad$ up to homeomorphism in adjoint types $E$.

\begin{lemma}\label{le:join.lens}
  For compact adjoint Lie groups $\bG$ of types $E_{6,7}$ the spaces $\bG/\Ad$ are
  \begin{equation*}
    E_7
    \ :\
    \bG/\Ad\cong \bD^4*\bR\bP^2
    \quad\text{and}\quad
    E_6
    \ :\ 
    \bG/\Ad\cong \bD^2*\bL^3,
  \end{equation*}
  with $\bL^3$ denoting one of the 3-dimensional \emph{lens spaces} $L_3$ of \cite[Example 2.43]{hatch_at}.
\end{lemma}
\begin{proof}
  As mentioned, this is a specialization of the discussion in step \Cref{item:th:conj.clss.fib:cycl} of the proof of \Cref{th:conj.clss.fib}. The $E_{6,7}$ rows of \cite[Table 3]{garnier:tel-03622954} describe the action of $\pi_1(\bG)$ on the simplex $\Delta:=\widetilde{\bG}/\Ad$ (where $\widetilde{\bG}\xrightarrowdbl{}\bG$ denotes the universal cover, as in step \Cref{item:th:conj.clss.fib:gen} of the above proof):
  \begin{equation*}
    E_7
    \ :\
    \text{product $\omega_a\omega_b\omega_c$ of three 2-cycles}
    \quad\text{and}\quad
    E_6
    \ :\ 
    \text{product $\omega_s\omega_t$ of two 3-cycles},
  \end{equation*}
  each permuting 6 of the vertices of the $\ell$-simplex $\Delta^{\ell}$, $\ell=7,6$ respectively. In each case the fixed-point space is the convex hull of the fixed vertices of $\Delta^{\ell}$ and the barycenters of the faces left invariant by the cycles $\omega_i$, so they are disks $\bD^4$ and $\bD^2$ for $\ell=7,6$ respectively. The links on those disks in $\Delta^\ell$ are
  \begin{itemize}[wide]
  \item for $\ell=7$, a 2-sphere acted on antipodally;

  \item and for $\ell=6$, a 3-sphere $\bS^3\cong \bS^1*\bS^1$ \cite[Proposition 2.23]{rs_pl_1982} acted upon by $\bZ/3$ via the latter's simultaneous rotation of the two circles. 
  \end{itemize}
  This settles $E_7$, while for $E_6$ note that the action just described is nothing but
  \begin{equation*}
    \bS^3\cong\left\{(z_1,z_2)\in \bC^2\ :\ \sum|z_i|^2=1\right\}
    \xmapsto{\quad\sigma\in \bZ/3\subset \bS^1\quad}
    (\sigma z_1,\sigma z_2)
    \in
    \bS^3,
  \end{equation*}
  whose quotient is by definition the lens space $L_3(1,1)$ of \cite[Example 2.43]{hatch_at}. In this case the parameters do not matter: per the classification \cite[Theorem, p.181]{zbMATH03194320} of 3-dimensional lens spaces, all $L_3(u,v)$, $u,v\in \bZ\setminus 3\bZ$ are mutually homeomorphic. 
\end{proof}

In the same spirit, we have the following description in for two other types not analyzed fully in the proof of \Cref{th:conj.clss.fib}. 

\begin{lemma}\label{le:array.along.lattice}
  For simple, compact Lie groups of types $C$ and $D$ the spaces $\bG/\Ad$ are homeomorphic to the following spaces, arrayed along the lattice of subgroups
  \begin{equation*}
    \pi_1(\bG)
    \le
    \pi_1\left(\Ad \bG\right),
  \end{equation*}
  with higher nodes corresponding to larger subgroups. 

  \begin{equation*}
    C_{n\ge 3}
    \ :\ 
    \begin{tikzpicture}[>=stealth,auto,baseline=(current  bounding  box.center)]
      \path[anchor=base] 
      (0,0) node (u) {$\bD^{(n-1)-\left\lfloor\frac{n-1}{2}\right\rfloor}*\bR\bP^{\left\lfloor\frac{n-1}{2}\right\rfloor}$}
      +(0,-1) node (d) {$\bD^n$}
      ;
      \draw[-] (u) to[bend left=0] node[pos=.5,auto] {$\scriptstyle $} (d);
    \end{tikzpicture}    
  \end{equation*}
  \begin{equation}\label{eq:le:array.along.lattice:d}
    D_{2n+1\ge 5}
    \ :\
    \begin{tikzpicture}[>=stealth,auto,baseline=(current  bounding  box.center)]
      \path[anchor=base] 
      (0,0) node (u) {$\bD^{n-1}*\bR\bP^{n+1}$}
      +(0,-1) node (d) {$\bD^{2n+1}$}
      +(0,-2) node (dd) {$\bD^{2n+1}$}
      ;
      \draw[-] (u) to[bend left=0] node[pos=.5,auto] {$\scriptstyle $} (d);
      \draw[-] (d) to[bend left=0] node[pos=.5,auto] {$\scriptstyle $} (dd);
    \end{tikzpicture}
    \qquad
    D_{2n\ge 4}
    \ :\ 
    \begin{tikzpicture}[>=stealth,auto,baseline=(current  bounding  box.center)]
      \path[anchor=base] 
      (0,0) node (u) {$\bD^{n}*\bR\bP^{n-1}$}
      +(-2,-1) node (dl) {$\bD^{n}*\bR\bP^{n-1}$}
      +(0,-1) node (dm) {$\bD^{2n}$}
      +(2,-1) node (dr) {$\bD^{n}*\bR\bP^{n-1}$}
      +(0,-2) node (dd) {$\bD^{2n}$}
      ;
      \draw[-] (u) to[bend right=6] node[pos=.5,auto] {$\scriptstyle $} (dl);
      \draw[-] (u) to[bend right=0] node[pos=.5,auto] {$\scriptstyle $} (dm);
      \draw[-] (u) to[bend left=6] node[pos=.5,auto] {$\scriptstyle $} (dr);
      \draw[-] (dl) to[bend right=6] node[pos=.5,auto] {$\scriptstyle $} (dd);
      \draw[-] (dm) to[bend right=0] node[pos=.5,auto] {$\scriptstyle $} (dd);
      \draw[-] (dr) to[bend left=6] node[pos=.5,auto] {$\scriptstyle $} (dd);
    \end{tikzpicture}
  \end{equation}
\end{lemma}
\begin{proof}
  We reprise the notation used in the proof of \Cref{th:conj.clss.fib} (e.g. writing $\widetilde{\bG}\xrightarrowdbl{}\bG$ for the universal cover of $\bG$).
  
  By \Cref{th:conj.clss.fib}\Cref{item:th:conj.clss.fib:bsc}, there is nothing to prove in the simply-connected cases (i.e. the bottom nodes of the various diagrams). 

  The cases of $\pi_1\cong \bZ/2=\braket{\omega}$ acting on the $r$-simplex $\bD^r\cong \widetilde{\bG}/\Ad$ ($r:=\rk\bG$) as a disjoint product of $k$ transpositions is also simple: that action is an involutive affine reflection across the $(r-k)$-dimensional affine span of $\mathrm{fix}(\omega)\subset\bD^r$; the quotient is thus
  \begin{equation}\label{eq:fixo.lk.fix.o}
    \mathrm{fix}~\omega
    *
    \left(\lk_{\bD^r}\left(\mathrm{fix}~\omega\right)/\braket{\omega}\right)
    \cong
    \bD^{r-k}*\bR\bP^{k-1}. 
  \end{equation}
  This settles (upon examining the relevant rows of \cite[Table 3]{garnier:tel-03622954})
  \begin{itemize}[wide]
  \item adjoint $C_n$, with $r:=n$ and $k:=\left\lfloor\frac{n+1}{2}\right\rfloor$;

  \item intermediate $D_{2n+1}$, with $r:=2n+1$ and $k:=2$ (in which case \Cref{eq:fixo.lk.fix.o} is again a disk, for projective 1-space is a circle);

  \item similarly, the intermediate $D_{2n}$ cases, with $r:=2n$ and $k:=2$ in one case and $k:=n$ in the two others. 
  \end{itemize}
  This leaves the adjoint types $D$ to handle (i.e. the top nodes in the 3-level diagrams \Cref{eq:le:array.along.lattice:d}).

  \begin{enumerate}[(I),wide]

  \item \textbf{: Adjoint $D_{2n+1\ge 5}$.} The fundamental group here is cyclic of order 4, generated by an element $\omega$ cyclic through 4 vertices of the simplex $\Delta^{2n+1}\cong \bD^{2n+1}\cong \widetilde{\bG}/\Ad$ and transposing $n-1$ other pairs of vertices.

    The square $\omega^2$ acts as a product of two disjoint transpositions on the vertices of a simplex $\Delta^3\subset \Delta^{2n+1}$. Quotienting first by $\braket{\omega^2}\cong \bZ/2$ produces another ball
    \begin{equation}\label{eq:dodd.interm.quot}
      \left(\Delta^3/\braket{\omega^2}\right)
      *
      \Delta^{2n-3}
      \cong
      \bD^3*\bD^{2n-3}
      \cong
      \bD^{2n+1}
    \end{equation}
    (as in the preceding discussion on intermediate types). In terms of the decomposition \Cref{eq:dodd.interm.quot}, $\omega$ acts on the two join factors $\bD^3$ and $\bD^{2n-3}$:
    \begin{itemize}
    \item on the latter by transposing its vertices in $n-1$ pairs (having identified $\bD^{2n-3}\cong \Delta^{2n-3}$);

    \item and on the former antipodally.
    \end{itemize}
    The resulting action of $\bZ/2\cong \braket{\omega}/\braket{\omega^2}$ on \Cref{eq:dodd.interm.quot} can thus be regrouped as
    \begin{equation}\label{eq:dodd.quotact}
      \left(\text{trivial on }\bD^{n-1}\right)*\left(\text{antipodal on }\bS^{n+1}\right),
    \end{equation}
    with the quotient claimed in the top left-hand node of \Cref{eq:le:array.along.lattice:d}.    
    
  \item \textbf{: Adjoint $D_{2n\ge 4}$.} This is the case discussed in part \Cref{item:th:conj.clss.fib:adjd} of the proof of \Cref{th:conj.clss.fib}, with
    \begin{equation*}
      (\bZ/2)^2
      \cong
      \pi_1(\bG)
      =
      \left\{1,\ \omega_{1,2,3}\right\}
      \quad
      \text{as in \Cref{eq:d2n.omegas}}. 
    \end{equation*}
    The preceding analysis goes through with appropriate modifications. First, the quotient by $\braket{\omega_1}\cong \bZ/2$ will be the disk
    \begin{equation}\label{eq:dev.interm.quot}
      \left(\Delta^3/\braket{\omega_1}\right)
      *
      \Delta^{2n-4}
      \cong
      \bD^3*\bD^{2n-4}
      \cong
      \bD^{2n},
    \end{equation}
    with the quotient $\pi_1/\braket{\omega_1}$ acting further on the two join factors
    \begin{itemize}[wide]
    \item by transposing the vertices of $\bD^{2n-4}\cong \Delta^{2n-4}$ in $n-2$ pairs;

    \item and operating on $\bD^3\cong \bD^1*\bS^1$ as
      \begin{equation*}
        \left(\text{identity on }\bD^1\right)*\left(\text{antipodal map on }\bS^1\right).
      \end{equation*}
    \end{itemize}
    Regrouping again (by analogy to \Cref{eq:dodd.quotact}), the action of the generating image of $\omega_2$ in $\pi_1/\braket{\omega_1}\cong \bZ/2$ on the quotient \Cref{eq:dev.interm.quot} is
    \begin{equation*}
      \left(\text{trivial on }\bD^{n}\right)*\left(\text{antipodal on }\bS^{n-1}\right).
    \end{equation*}
    Once more, the quotient is as in the right-hand top node of \Cref{eq:le:array.along.lattice:d}.
  \end{enumerate}  
\end{proof}



As a topical aside (relating to some of the specifics of how Weyl groups act on tori of compact simple Lie groups), recall the crucial observation in the proof of \Cref{pr:prod.protor.ss} (part \Cref{item:pr:prod.protor.ss:sc}) that for simply-connected compact simple $\bG$ the map
\begin{equation*}
  \bT
  \quad
  \overset{\text{maximal torus}}{\le}
  \quad
  \bG
  \xrightarrowdbl{\quad}
  \bG/\Ad
  \quad\overset{\text{\cite[Proposition IV.2.6]{btd_lie_1995}}}{\cong}\quad
  \bT/W
\end{equation*}
splits continuously. This is certainly \emph{not} always the case: never, for instance, for non-simply-connected type-$A$ groups:

\begin{proposition}\label{pr:afy.fundom}
  Let $n\in \bZ_{\ge 2}$, $\bG$ a non-simply-connected compact simple type-$A_n$ Lie group, and $\bT\subset \bG$ a maximal torus.

  The surjection
  \begin{equation}\label{eq:ft2tw}
    \begin{tikzpicture}[>=stealth,auto,baseline=(current  bounding  box.center)]
      \path[anchor=base] 
      (0,0) node (l) {$\ft:=Lie(\bT)$}
      +(2,.5) node (u) {$\bT$}
      +(4,0) node (r) {$\bT/W,$}
      +(6.5,0) node () {$W:=\text{Weyl group}$}
      ;
      \draw[->>] (l) to[bend left=6] node[pos=.5,auto] {$\scriptstyle \exp$} (u);
      \draw[->>] (u) to[bend left=6] node[pos=.5,auto] {$\scriptstyle $} (r);
      \draw[->>] (l) to[bend right=6] node[pos=.5,auto,swap] {$\scriptstyle \pi$} (r);
    \end{tikzpicture}
  \end{equation}
  does not split continuously.
\end{proposition}

For a compact, connected, simple Lie group $\bG$ and a maximal torus $\bT\subset \bG$ we denote by
\begin{equation}\label{eq:ft}
  \ft:=Lie(\bT)\supset F_{\bT}
  \xrightarrow[\quad\cong\quad]{\quad\exp\quad}
  \exp F_{\bT}
  \subset
  \bT
\end{equation}
any one of the spaces $F_Y$ of \cite[Theorem 2.3.5]{2409.16483v1}. The image $\exp F_{\bT}$ is a \emph{fundamental domain} for the Weyl-group action in the sense of \cite[\S 1.2]{2409.16483v1}:
\begin{enumerate}[(a),wide]
\item closed and connected;
\item\label{item:int.dist.orb} with no distinct elements in its interior belonging to the same orbit;
\item and whose $W$-translates cover $\bT$. 
\end{enumerate}
This is very much in line with the way the phrase is used in many other sources (\cite[\S III.1, p.100]{kob_ell_2e_1993}, say; see also \cite[\S 3]{2308.11997v1}), but note that the fundamental domains of \cite[\S 1.12, p.22]{hmph_cox} entail more: there, no distinct points lie on the same orbit period (so condition \Cref{item:int.dist.orb} above is strengthened). The discrepancy between the two competing notions is in a sense at the heart of the phenomenon recorded in \Cref{pr:afy.fundom}.

\pf{pr:afy.fundom}
\begin{pr:afy.fundom}
  In the spirit of \Cref{eq:ft}, we substitute `$\bT$' for `$Y$' as a subscript for the group $\Omega_Y$ of \cite[Proposition 2.3.4]{2409.16483v1}: it is here denoted by $\Omega_{\bT}\cong \pi_1(\bG)$. It will be enough to argue that $F_{\bT}\subset \ft$ contains distinct points in the same Weyl orbit: a continuous right inverse $\iota$ of \Cref{eq:ft2tw} can be assumed to map any one point in
  \begin{equation*}
    \pi\left(\text{interior }\overset{\circ}{F_{\bT}}\right)
    \subset \bT/W
  \end{equation*}
  back into $\overset{\circ}{F_{\bT}}$, whence
  \begin{equation*}
    \iota|_{\pi\left(\overset{\circ}{F_{\bT}}\right)}
    =
    \left(\pi|_{\overset{\circ}{F_{\bT}}}\right)^{-1}
  \end{equation*}
  and $\pi$ would have to be injective on $F_{\bT}$.   
  
  Let $1\ne \omega\in \Omega_{\bT}$. By \cite[Proposition 2.3.3]{2409.16483v1}, the polytope $F\subseteq F_{\bT}$ constructed in \cite[Proposition 2.2.3]{2409.16483v1} (and denoted there by $F_{P^{\vee}}$) intersects its iterate $\omega F$ along a \emph{facet} $F'$ (i.e. maximal proper face \cite[\S 2.6]{grnbm_polyt}). We thus have 
  \begin{equation*}
    \omega F_{\bT}\cap F_{\bT}
    \supseteq
    \omega F\cap F
    =
    F',
  \end{equation*}
  and it remains to argue that $\omega$ cannot fix $F'$ pointwise. This follows from the type-$A$ row of \cite[Table 3]{garnier:tel-03622954}: $\omega$ is an affine operator permuting the $n+1$ vertices of the $n$-simplex
  \begin{equation*}
    \Delta^n\cong \widetilde{\bT}/W
    ,\quad
    \left(\text{$\widetilde{\bT}\subset \widetilde{\bG}$ a maximal torus of the universal cover $\widetilde{\bG}\xrightarrowdbl{}\bG$}\right).
  \end{equation*}
  $\omega$ being non-trivial, it operates on the $n+1$ vertices of $\Delta^n$ as a product
  \begin{equation*}
    \omega=\sigma_1\cdot \sigma_2\cdots \sigma_k
    ,\quad
    1\le k<n
    ,\quad
    \sigma_i\text{ an $\frac{n+1}{k}$-cycle}.
  \end{equation*}
  Each $\sigma_i$ cycles through the $\frac{n+1}{k}$ vertices of a $\frac{n+1-k}{k}$-simplex $\Delta_i$, and the affine automorphism of $\aspn\Delta^n$ induced by $\omega$ is an order-$\frac{n+1}{k}$ rotation around the affine $(k-1)$-plane
  \begin{equation*}
    \aspn \left\{\text{barycenter of }\Delta_i\ :\ 1\le i\le k\right\}.
  \end{equation*}
  In particular, the fixed-point set of $\omega$ is $(k-1)$-dimensional; certainly, then, $\omega$ cannot fix an $F$-facet (of dimension $n-1>k-1$) pointwise.  
\end{pr:afy.fundom}

\addcontentsline{toc}{section}{References}

\begin{thebibliography}{10}

\bibitem{bourb_lie_1-3}
Nicolas Bourbaki.
\newblock {\em Lie groups and {L}ie algebras. {C}hapters 1--3}.
\newblock Elements of Mathematics (Berlin). Springer-Verlag, Berlin, 1998.
\newblock Translated from the French, Reprint of the 1989 English translation.

\bibitem{btd_lie_1995}
Theodor Br{\"o}cker and Tammo tom Dieck.
\newblock {\em Representations of compact {Lie} groups. {Corrected} reprint of
  the 1985 orig}, volume~98 of {\em Grad. Texts Math.}
\newblock New York, NY: Springer, corrected reprint of the 1985 orig. edition,
  1995.

\bibitem{zbMATH03194320}
E.~J. Brody.
\newblock The topological classification of the lens spaces.
\newblock {\em Ann. Math. (2)}, 71:163--184, 1960.

\bibitem{zbMATH01714503}
John~L. Bryant.
\newblock Piecewise linear topology.
\newblock In {\em Handbook of geometric topology}, pages 219--259. Amsterdam:
  Elsevier, 2002.

\bibitem{2501.06840v1}
Alexandru Chirvasitu, Ilja Gogi\'c, and Mateo Toma\v{s}evi\'c.
\newblock Continuous spectrum-shrinking maps and applications to preserver
  problems, 2025.
\newblock \url{http://arxiv.org/abs/2501.06840v1}.

\bibitem{MR1247532}
Albrecht Dold.
\newblock A simple proof of the {J}ordan-{A}lexander complement theorem.
\newblock {\em Amer. Math. Monthly}, 100(9):856--857, 1993.

\bibitem{ds_linop-1_1958}
Nelson Dunford and Jacob~T. Schwartz.
\newblock Linear operators. {I}. {General} theory. ({With} the assistence of
  {William} {G}. {Bade} and {Robert} {G}. {Bartle}).
\newblock Pure and {Applied} {Mathematics}. {Vol}. 7. {New} {York} and
  {London}: {Interscience} {Publishers}. xiv, 858 p. (1958)., 1958.

\bibitem{2308.11997v1}
J\"{u}rgen Elstrodt.
\newblock What is a fundamental domain?, 2023.
\newblock \url{http://arxiv.org/abs/2308.11997v1}.

\bibitem{fp_cell}
Rudolf Fritsch and Renzo~A. Piccinini.
\newblock {\em Cellular structures in topology}, volume~19 of {\em Camb. Stud.
  Adv. Math.}
\newblock Cambridge etc.: Cambridge University Press, 1990.

\bibitem{fh_rep-th}
William Fulton and Joe Harris.
\newblock {\em Representation theory. {A} first course}, volume 129 of {\em
  Grad. Texts Math.}
\newblock New York etc.: Springer-Verlag, 1991.

\bibitem{garnier:tel-03622954}
Arthur Garnier.
\newblock {\em {Equivariant cellular models in Lie theory}}.
\newblock Theses, {Universit{\'e} de Picardie Jules Verne}, December 2021.

\bibitem{2409.16483v1}
Arthur Garnier.
\newblock Fundamental polytope for the weyl group acting on a maximal torus of
  a compact lie group, 2024.
\newblock \url{http://arxiv.org/abs/2409.16483v1}.

\bibitem{grnbm_polyt}
Branko Gr\"unbaum.
\newblock {\em Convex polytopes}, volume 221 of {\em Graduate Texts in
  Mathematics}.
\newblock Springer-Verlag, New York, second edition, 2003.
\newblock Prepared and with a preface by Volker Kaibel, Victor Klee and
  G\"unter M.\ Ziegler.

\bibitem{hatch_at}
Allen Hatcher.
\newblock {\em Algebraic topology}.
\newblock Cambridge: Cambridge University Press, 2002.

\bibitem{helg_dglgssp}
Sigurdur Helgason.
\newblock {\em Differential geometry, {Lie} groups, and symmetric spaces},
  volume~80 of {\em Pure Appl. Math., Academic Press}.
\newblock Academic Press, New York, NY, 1978.

\bibitem{hm5}
Karl~H. Hofmann and Sidney~A. Morris.
\newblock {\em The structure of compact groups. {A} primer for the student. {A}
  handbook for the expert}, volume~25 of {\em De Gruyter Stud. Math.}
\newblock Berlin: De Gruyter, 5th edition edition, 2023.

\bibitem{hmph_1980}
James~E. Humphreys.
\newblock {\em Introduction to {Lie} algebras and representation theory. 3rd
  printing, rev}, volume~9 of {\em Grad. Texts Math.}
\newblock Springer, Cham, 1980.

\bibitem{hmph_cox}
James~E. Humphreys.
\newblock {\em Reflection groups and {Coxeter} groups}, volume~29 of {\em Camb.
  Stud. Adv. Math.}
\newblock Cambridge: Cambridge University Press, 1992.

\bibitem{hjjm_bdle}
D.~Husem\"{o}ller, M.~Joachim, B.~Jur\v{c}o, and M.~Schottenloher.
\newblock {\em Basic bundle theory and {$K$}-cohomology invariants}, volume 726
  of {\em Lecture Notes in Physics}.
\newblock Springer, Berlin, 2008.
\newblock With contributions by Siegfried Echterhoff, Stefan Fredenhagen and
  Bernhard Kr\"{o}tz.

\bibitem{hus_fib}
Dale Husemoller.
\newblock {\em Fibre bundles}, volume~20 of {\em Graduate Texts in
  Mathematics}.
\newblock Springer-Verlag, New York, third edition, 1994.

\bibitem{iw}
Kenkichi Iwasawa.
\newblock On some types of topological groups.
\newblock {\em Ann. of Math. (2)}, 50:507--558, 1949.

\bibitem{kap}
Irving Kaplansky.
\newblock {\em Infinite abelian groups}.
\newblock University of Michigan Press, Ann Arbor, 1954.

\bibitem{kob_ell_2e_1993}
Neal Koblitz.
\newblock {\em Introduction to elliptic curves and modular forms.}, volume~97
  of {\em Grad. Texts Math.}
\newblock New York: Springer-Verlag, 2. ed. edition, 1993.

\bibitem{LiZhang}
Chi-Kwong Li and Fuzhen Zhang.
\newblock Eigenvalue continuity and {G}er\v{s}gorin's theorem.
\newblock {\em Electron. J. Linear Algebra}, 35:619--625, 2019.

\bibitem{mack-unit}
George~W. Mackey.
\newblock {\em The theory of unitary group representations}.
\newblock Chicago Lectures in Mathematics. University of Chicago Press,
  Chicago, Ill.-London, 1976.
\newblock Based on notes by James M. G. Fell and David B. Lowdenslager of
  lectures given at the University of Chicago, Chicago, Ill., 1955.

\bibitem{MR210096}
H.~R. Morton.
\newblock Symmetric products of the circle.
\newblock {\em Proc. Cambridge Philos. Soc.}, 63:349--352, 1967.

\bibitem{mnk}
James~R. Munkres.
\newblock {\em Topology}.
\newblock Prentice Hall, Inc., Upper Saddle River, NJ, 2000.
\newblock Second edition of [ MR0464128].

\bibitem{rz_prof_2e_2010}
Luis Ribes and Pavel Zalesskii.
\newblock {\em Profinite groups.}, volume~40 of {\em Ergeb. Math. Grenzgeb., 3.
  Folge}.
\newblock Berlin: Springer, 2nd ed. edition, 2010.

\bibitem{rom_lalg_3e_2008}
Steven Roman.
\newblock {\em Advanced linear algebra}, volume 135 of {\em Grad. Texts Math.}
\newblock New York, NY: Springer, 3rd ed. edition, 2008.

\bibitem{rs_pl_1982}
C.~P. Rourke and B.~J. Sanderson.
\newblock {\em Introduction to piecewise-linear topology.}, volume~69 of {\em
  Ergeb. Math. Grenzgeb.}
\newblock Springer-Verlag, Berlin, revised reprint of the 1972 original
  edition, 1982.

\bibitem{rud_lc}
Walter Rudin.
\newblock {\em Fourier analysis on groups}.
\newblock Wiley Classics Library. John Wiley \& Sons, Inc., New York, 1990.
\newblock Reprint of the 1962 original, A Wiley-Interscience Publication.

\bibitem{sep_cpct-lie}
Mark~R. Sepanski.
\newblock {\em Compact {Lie} groups}, volume 235 of {\em Grad. Texts Math.}
\newblock New York, NY: Springer, 2007.

\bibitem{stein_chev}
Robert Steinberg.
\newblock {\em Lectures on {Chevalley} groups}, volume~66 of {\em Univ. Lect.
  Ser.}
\newblock Providence, RI: American Mathematical Society (AMS), 2016.

\bibitem{td_alg-top}
Tammo tom Dieck.
\newblock {\em Algebraic topology}.
\newblock EMS Textb. Math. Z{\"u}rich: European Mathematical Society (EMS),
  2008.

\bibitem{wil_top}
Stephen Willard.
\newblock {\em General topology}.
\newblock Dover Publications, Inc., Mineola, NY, 2004.
\newblock Reprint of the 1970 original [Addison-Wesley, Reading, MA;
  MR0264581].

\end{thebibliography}

\def\polhk#1{\setbox0=\hbox{#1}{\ooalign{\hidewidth
  \lower1.5ex\hbox{`}\hidewidth\crcr\unhbox0}}}

\Addresses

\end{document}